\documentclass[preprint,9pt]{elsarticle}

\usepackage{graphics}
\usepackage{amsmath,amsfonts,amssymb,latexsym}
\usepackage{natbib}
\usepackage{graphicx}
\usepackage{float}
\usepackage{geometry}
\newtheorem{theorem}{Theorem}[section]

\newtheorem{proposition}[theorem]{Proposition}

\newtheorem{remark}[theorem]{Remark}
\numberwithin{equation}{section}

\newenvironment{proof}[1][Proof]{\begin{trivlist}
\item[\hskip \labelsep {\bfseries #1}]}{\end{trivlist}}

\newenvironment{example}[1][Example]{\begin{trivlist}
\item[\hskip \labelsep {\bfseries #1}]}{\end{trivlist}}

 {\left\lbrace\begin{array}{@{}l@{}}}%
 {\end{array}\right.}
 {\left\lbrace\begin{eqnarray}{@{}l@{}}}%
 {\end{eqnarray}\right.}

\newcommand{\R}{\mathbb{R}}

\journal{}

\begin{document}

\begin{frontmatter}

\title{Rosenbrock-type methods applied to discontinuous differential systems}

\author[focal]{Marco Berardi}
\ead{berardi@dm.uniba.it}
\address[focal]{Dipartimento di Matematica, Universit\`a degli Studi di Bari, Via E. Orabona 4, I-70125, Bari, Italy}

\begin{abstract}
In this paper we will study the  numerical solution of a discontinuous differential system by a Rosenbrock method. We will also focus on \emph{one-sided} approach in the context of Rosenbrock schemes, and we will suggest a technique based on the use of continuous extension, in order to locate the event point, with an application to discontinuous singularly perturbed systems. 
\end{abstract}


\end{frontmatter}

\section{Introduction}\label{Sec1}
In this paper we will study a class of one-step schemes for ODEs, i.e. the class of Rosenbrock methods, in the context of ODEs with discontinuous right-hand side, with particular reference to singularly perturbed discontinuous ODEs. The class of problems we are dealing with is generally expressed in the form: 

\begin{equation}\label{system1}
x'=f(x)= \left\{
\begin{array}{rl}
 f_1(x)& \mbox{when } x \in R_1 \\
 f_2(x)& \mbox{when } x \in R_2 \\
\end{array},
\right.
\end{equation}
for $t\ge \bar{t}$ and $x(\bar{t})=v$ (see \cite{Aubin.Cellina},
\cite{Filippov},  \cite{Utkin}, \cite{Dieci.Lopez.survey}).
 The state space $\R^n$ is split
(locally) into two subspaces $R_1$ and $R_2$ by a surface
$\Sigma$ such that $\R^n=R_1 \cup\Sigma \cup R_2$.  The surface
is defined by a scalar {\em event} function $h:\R^n \to \R$, so that
the subspaces $R_1$ and $R_2$, and $\Sigma$, are characterized as
\begin{equation}\label{Sigmas}
\Sigma= \left\{x\in \R^n |\ h(x)=0 \right\},\  R_1= \left\{ x\in \R^n
|\ h(x)<0 \right\},\ R_2= \left\{x\in \R^n  | \ h(x)>0 \right\}.
\end{equation}

When solving numerically such a discontinuous system, at each integration step the occurrence of a discontinuity is checked. What we do in practice, at a general $(n+1)$-th step, is to check the sign of $h(x_n)\cdot h(x_{n+1})$: if this product is greater than zero, then we continue using the same vector field ($f_1$, or $f_2$) as in the $n$-th step. Otherwise, if the product is less than zero, this means that we need to switch to the other vector field, or to a sliding vector field; in the \emph{event driven} approach, that is the approach we will follow, this switching of the vector field requires the accurate computation of the \emph{event}, i.e. the state in which the event function $h$ vanishes; as a matter of fact, our coverage will mainly focus on the problems of the event detection. \\

Event-driven methods can be applied just when there are finitely many event points. This class of methods is widely used (see \cite{Berardi.Lopez, Dieci.Lopez.2010, Mannshardt, Piiroinen.Kuznetsov, Shampine.Thompson, Stewart2}); it seems to be particularly suited since it has been proved (see for example \cite{Dieci.Lopez.survey}, and Section 3 of this paper) that the order of any method (both explicit and implicit, both
 one-step and multi-step) falls to one when discontinuity occurs and the event is not accurately located. Hence, we could say that the event location reduces the stiffness of the problem due to the discontinuity. \\ 

Another important issue of this paper is the so-called \emph{one-sided approach}. As a matter of fact, when we numerically solve a discontinuous system of the form (\ref{system1}), we would wish that the vector fields $f_1$ and/or $f_2$ would be defined also beyond $\Sigma$. Neverthless, in certain situations, the vector field is not defined everywhere because the model is designed to be applicable only in certain regions of the state space. This particular type of systems, sometimes called systems with \emph{model singularities}, can be found, for example, in \cite{Najafi.Azil.Nikoukhah, Esposito.Kumar.2004, Esposito.Kumar.2007}.

\smallskip\noindent
As simple example (taken from \cite{Najafi.Azil.Nikoukhah}) 
we may consider 

\begin{equation}\label{Najafi}
  x'=
\left\{
\begin{array}{rl}
 x\sqrt{1-t}   & \mbox{when }  t\le 1 \\
 0& \mbox{when } t>1  \\
\end{array},
\right.  
\end{equation}
where one of the vector fields is not defined when $t>1$. 
In general, every differential system whose right-hand side involves square roots, logarithms, inverse trigonometric functions, possesses a \emph{singularity}, where by singularity we mean the region of the state space in which the derivative function $f_i$ in (\ref{system1}) is undefined. The events that occur in a neighborhood of a model singularity are often referred to as \emph{unilateral events}, which means that they should be detected without allowing the numerical solution to trespass the event itself. \\
For this reason, we cannot use implicit methods, since they require evaluating $f(x)$ at possibly singular endpoints. This is the rationale of the approach proposed in Section 3.

We know from Filippov theory (see for example \cite{Dieci.Lopez.2011}) what can happen when the solution 'hits' the surface $\Sigma$. 
We will focus on the event location: what will happen after the event is not our present concern.\\
\smallskip

Finally, we will focus on a particular class of discontinuous ODEs: the discontinuous singularly perturbed systems. The presence of a singular perturbation in a discontinuous system arises in a lot of applications (see \cite{AlvarezGallego.SilvaNavarro, Fridman.2,  Fridman.SotoCota.Loukianov.Canedo, Heck, Heck.Haddad, Sieber.Kowalczyk}). \\
Of course these systems are difficult to solve, because of the discontinuity, and, most of all, because of the singular perturbation, which introduces a strong stiffness in the problem. The use of Rosenbrock methods is particularly convenient for these systems, because of the low computational effort, and of the good stability properties. This will be shown in Section 5. 

\section{Rosenbrock methods}
As we shall see in the following, Rosenbrock schemes turn out to be very advantageous in the context both of one-sided methods and of discontinuous singularly perturbed problems.\\
Rosenbrock methods come out from the linearization of diagonally implicit Runge-Kutta methods (see \cite{Hairer.Wanner.2008}). They preserve good stability properties typical of implicit schemes; on the other hand, an $s$-stage Rosenbrock method requires a lower computational effort, since just $s$ linear systems must be solved. [Frequently, as we shall see later, we can handle the method in such a way that the $s$ linear systems to be solved have the same coefficient matrix.]
\\
For an autonomous system, like the one in (\ref{system1}), an  {\em $s$-stage Rosenbrock method } has the following expression:
\begin{equation}\label{aut_Rosenb}
\begin{array}{l}
x_1 = x_0 + \sum_{i = 1}^{s}b_i k_i \\
k_i = \tau f\left( x_0 + \sum_{i = 1}^{i-1} \alpha_{ij}k_j \right) + \tau J  \sum_{i = 1}^{i} \gamma_{ij} k_j
\end{array},
\end{equation}
where $\alpha_{ij}$, $\gamma_{ij}$ and $b_i$ are the coefficients of the method, and $J = f'(x_0)$.  \\
Each stage of this method requires the solution of $s$- linear systems with unknowns $k_i$, and with matrix $I - \tau \gamma_{ii}J$. Most popular Rosenbrock methods (see, for example, \cite{Roche}) set $\gamma_{ii} = \gamma$, for every $i = 1, ... ,s$. This position is computationally advantageous, since all the matrices are equal and we only need one LU-factorization per step. Particularly interesting is the use of Rosenbrock methods in a singularly perturbed system with discontinuous right-hand side, as we shall see in section 6.
\\

\subsection{An example of order reduction}
The phenomenon of order reduction in discontinuous differential systems, when the event is not accurately located, is well-known in the literature. Gear and \O sterby, for instance, deduced the order reduction in Predictor-Corrector methods (see \cite{Gear.Osterby}). On the other hand, according to the pioneering work of Mannshardt, (see \cite{Mannshardt}), Lopez and Dieci (see \cite{Dieci.Lopez.survey}) have accurately computed the global error for explicit Euler method; they have shown that there are two contributions to the $O(\tau)$ term in the global error: the first depends on the fact that Euler's method is a first order method, while the second contribution comes directly from the jump and does not depend on the order of the method. \\ 
We will give an example of the order reduction of a $2$-stages Rosenbrock method, following the approach of \cite{Dieci.Lopez.survey} for the explicit Euler method. \\
What we want to do is to evaluate the leading term of the local truncation error in the discontinuity interval. For simplicity of notation, we assume $t_0 = 0$, but this assumption is not restrictive. Naturally, localizing assumption applies, i.e. $x_0 = x(0)$; also, we call $\xi_1$ the instant of the event. \\
Our 2-stage Rosenbrock methods, applied to the problem (\ref{system1}), reads:
\begin{equation}\label{Rosenbrock_2_step}
\begin{array}{l}
x_1 = x_0 + \frac{3}{2} k_1 + \frac{1}{2}k_2, \quad \textrm{where}\\
(I-\tau \gamma J) k_1 = \tau f( x_0 ),\\
(I-\tau \gamma J) k_2 = \tau f( x_0 + k_1).
\end{array}
\end{equation}

Calling $A$ the inverse of the matrix $(I-\tau \gamma J)$, and assuming that both $x_0$ and $x_0 + \tau A f_1(x_0)$ are in $R_1$, the method (\ref{Rosenbrock_2_step}) can be written as 
\begin{equation}
\label{alternative_formulation}
x_1  = x_0 + \frac{3}{2} \tau A f_1(x_0) + \frac{1}{2} \tau A f_1 \big(x_0 + \tau A f_1(x_0)) - \tau A^2 f_1(x_0 \big).
\end{equation}

A Taylor's expansion of the exact solution up to the first order gives:
\begin{equation}
x(\tau) = x_0 + \xi_1f_1(x_0) + (\tau - \xi_1)f_2(x(\xi_1)) + O(\tau^2),
\end{equation}
while the local truncation error ($l.t.e.$) at $t_1 = \tau$ becomes: 
\begin{equation} \label{ISABELLA}
\begin{split}
\emph{l.t.e.} &=  \xi_1f_1(x_0) + (\tau- \xi_1)f_2(x(\xi_1)) + O(\tau^2) - \frac{3}{2} k_1 - \frac{1}{2} k_2 =   \\
 &= \xi_1 f_1(x_0) + (\tau- \xi_1)f_2(x(\xi_1)) + O(\tau^2) - \frac{3}{2}A \tau f_1(x_0)  \\
 & \quad - \frac{1}{2}A \tau f_1(x_0+ A\tau f_1(x_0)) + \tau A^2 f_1(x_0) + O(\tau^2).  
\end{split}
\end{equation}

A first order approximation of our second-order method is sufficient to highlight the proportionality of the local truncation error (in the discontinuity interval) with the jump, by means of a positive proportionality factor that is less than $\tau$. 
Thus, denoting by $\rho(\gamma \tau J)$ the spectral radius of matrix $\gamma \tau J$, if $\rho(\gamma \tau J) < 1$, then 
\begin{equation} \label{Marial}
A = I + \gamma \tau J + O(\tau^2).
\end{equation}
On the other hand, if $v := A f_1(x_0)$, we can expand the term $f_1(x_0 + \tau v)$ as
\begin{equation} \label{sviluppo_1}
f_1(x_0 + \tau v) = f_1(x_0) + \tau J_* v, 
\end{equation}
where $J_*$ is the Jacobian of $f_1$ evaluated in some point in $R_1$. \\
By means of (\ref{Marial}) and (\ref{sviluppo_1}) we get that (\ref{ISABELLA}) becomes
\begin{equation*} 
\begin{split}
\emph{l.t.e.} &=  \xi_1f_1(x_0) + (\tau- \xi_1)f_2(x(\xi_1)) + O(\tau^2) - \frac{3}{2}\tau f_1(x_0)  \\
 & \quad - \frac{1}{2}\tau f_1(x_0) + \tau f_1(x_0) + O(\tau^2) \\
 &= \xi_1f_1(x_0) + (\tau- \xi_1)f_2(x(\xi_1)) + O(\tau^2) -\tau f_1(x_0) + O(\tau^2) \\
 &= (\tau - \xi_1) \big[ f_2(x(\xi_1)) - f_1(x_0) \big] + O(\tau^2). 
\end{split}
\end{equation*}
Finally, since
\[
f_1(x_0) = f_1\left( x \left( \xi_1 \right) \right) + \left( x_0 - x\left( \xi_1 \right) \right) J_{f_1}\left(x\left(\xi_1\right)\right) + O(\tau^2),
\]
and
\[
x_0 = x\left(\xi_1\right) - \xi_1 f_1\left( x\left( \xi_1 \right) \right) + O(\tau^2),
\]
we get that
\begin{equation*} 
\emph{l.t.e.} = (\tau - \xi_1) \big[ f_2\left(x\left(\xi_1\right)\right) - f_1\left(x\left(\xi_1\right)\right) \big] + O(\tau^2). 
\end{equation*}

We have thus shown the order reduction of this second order Rosenbrock method when applied to a discontinuous differential equation like (\ref{system1}).\\
\\
\subsection{Continuous extension of Rosenbrock methods}
In general, any numerical method for ODEs provides an approximation of the solution at certain mesh points. On the other hand, in certain applications (graphics, delay differential equations, initial value problems with driving conditions), these discrete values are not enough. We could need a \emph{dense output}, i.e., a numerical solution defined in each point $t$ in the integration interval, $[0, T]$. The event location, in the context of discontinuous differential equations, is one of the cases in which the idea of continuous extension can be advantageous. Naturally, dense output formulas can be found in different ways: first of all, trivially, by piecewise linear or cubic interpolants. Nevertheless, for Runge-Kutta methods there are more specific manners: perturbed collocated solutions (see \cite{Norsett.Wanner}), and the classical \emph{continuous extension} proposed in \cite{Zennaro}. It is known that, for the classes of Gauss and Radau formulas, these collocation methods have almost half the order of the method itself. Some authors (see \cite{Enright.J.N.T_old} and \cite{Gladwell.S.B.B.}) add some extra stages to achieve an accuracy of $O(\tau^p)$, where $p$ is the order of this method. \\
Now, in the context of discontinuous differential systems, dense output for Runge-Kutta formulas have been proposed both in \cite{Enright.J.N.T}, and also in \cite{Dieci.Lopez.2011, Hairer.Wanner.2008}. In this latter paper, for example, authors are able to find the event point simply by seeking the root of a second order continuous extension of the explicit midpoint rule. In this way, further evaluations of function $f$ are avoided, and just a second-order polynomial has to be updated. 
It is noteworth that the continuous extensions of Rosenbrock methods have been proposed also in the context of DAEs: in \cite{Xin.Xiaoqiu.Degui} two Rosenbrock methods are proposed: a 4-stage Rosenbrock scheme of order 3, and a 3-stage Rosenbrock method of order 2.  \\
Here, we focus on the continuous extension of Rosenbrock methods, whose theory has been investigated -in the smooth case- in \cite{Ostermann}. 
The theory of Ostermann, who threads the same path as Zennaro (see \cite{Zennaro}) for continuos extensions of Runge-Kutta methods, gives a technique for evaluate the solution outside of the mesh, i.e. for approximate $x(t_0 + \sigma), \forall \sigma \in [0, \tau]$. 
We define
\begin{equation} 
\label{cont_ext}
X(\theta) = x_0 + \sum_{i = 1}^{s} b_i(\theta) k_i, \quad \textrm{for} \quad 0 \leq \theta  \leq 1,
\end{equation}
a \emph{continuos extension} of Rosenbrock method (\ref{aut_Rosenb}), where the functions $b_i(\theta)$ are polynomials and satisfy: 
\begin{description}
\item[i)] $b_i(0) = 0$;
\item[ii)] $b_i(1) = b_i$.
\end{description}
We point out that (\ref{cont_ext}) only needs already known facts (the stages $k_i$) and can be evaluated cheaply, since just a low-order polynomial in $\theta$ must be updated. \\
Let us denote by $[\cdot]$ the function \emph{integer part}; thus we know, by \cite{{Ostermann}}, that every Rosenbrock method of order $p$ possesses at least one continuous extension of order $q = \left[ \frac{p+1}{2} \right]$, of the form (\ref{cont_ext}); moreover the polynomials $b_i(\sigma)$ are defined in theorem 1 of \cite{Ostermann}, and have degree at most $q$. Nevertheless, Ostermann's theory does not exclude the chance of finding a continuous extension that retains the order of the underlying method. \\
\smallskip

As a matter of fact, we are going to use a second order continuous extension of a second-order Rosenbrock method. The advantage of this choice, with respect to the choice in \cite{Xin.Xiaoqiu.Degui}, is that just a 2-stages method is needed to get a second-order approximation. \\
The Rosenbrock scheme we use has been introduced in \cite{Verwer.1999} and reads 
\begin{subequations}
\label{eqn:Isabella3}
\begin{align}
& x_1(\tau) = x_0 + \frac{3}{2} k_1 + \frac{1}{2} k_2, \\
& (I - \gamma \tau J) k_1 = \tau f(x_0), \label{eqn:sub} \\
& (I - \gamma \tau J) k_2 = \tau f(x_0 + k_1) - 2 \tau k_1,
\end{align}
\end{subequations}
for $\gamma = 1 - \frac{\sqrt{2}}{2}$. 
Now, a second-order continuos extension of this method has been proposed in \cite{Savcenco}: 
\begin{subequations}
\label{ISABELLA_cleavage}
\begin{align}
X_1(\theta) &= x_0 + \frac{1}{2(1 - 2\gamma)} b_1(\theta) k_1 + \frac{1}{2(1 - 2\gamma)} b_2(\theta)  k_2, \\
b_1(\theta) &=  \theta^2 + (2 - 6 \gamma)\theta , \\
b_2(\theta) &=  \theta^2 - 2 \gamma \theta.
\end{align}
\end{subequations}
\smallskip
\begin{remark}
Let us assume we are integrating by method (\ref{eqn:Isabella3}). When an event has occurred, i.e. when $h(x_0)h(x_1)<0$, then we need to compute accurately the state vector $x(\overline{\tau})$ such that $h(x(\overline{\tau})) = 0$. This is simply done by computing the root of the scalar function $H(\tau) := h(x_1(\tau))$, where $x_1(\tau)$  is the numerical solution of our Rosenbrock method in (\ref{eqn:Isabella3}). This computation is accomplished by a classical root-finding routine, such as secant or bisection method. \\
Every root-finding routine will produce a sequence $\tau_i$. Of course, the computation of each term of the sequence requires of updating the internal stages $k_1$ and $k_2$, since the internal stages depend on the step size. This can be avoided by using the continuous extension (\ref{ISABELLA_cleavage}), that has the great advantage (as we shall see in numerical tests, in the last section) of preserving the order of the method. Thus, we are going to compute the root of the new function $\widehat{H}(\theta) := h(X_1(\theta))$, for $0 \leq \theta \leq 1$, with a great computational saving. 
\end{remark}
Finally, we are going to see in the next section that the continuous extension can be very useful also in order to give one-sided conditions.

\section{One-sided Rosenbrock methods}
Further to what we said in the introduction about model singularities, we provide an interesting example, proposed in 
\cite{Esposito.Kumar.2007}, of a planar two-link robotic manipulator with workspace limitations (see Figure \ref{example_2_arms}). The dynamics of this system are described by the following system of ODEs:
\begin{subequations}
\begin{align}
\theta_1' = \omega_1; \\
\theta_2' = \omega_2,
\end{align}
\end{subequations}
for certain functions $\omega_1$ and $\omega_2$. Now, we could express $\theta_1$ and $\theta_2$ as functions of $x$ and $y$, which denote the position of the two arms in the plane:

\begin{align}
\theta_1 &= {\rm arctan}\left(\frac{y}{x}\right) - {\rm arccos}\left(\frac{x^2 + y^2 + l_1^2 + l_2^2}{2l_1 \sqrt{x^2 + y^2}}\right); \\
\theta_2 &= {\rm arctan}(\frac{y - l_1 sin(\theta_1)}{x - l_1 cos(\theta_1)}) - \theta_1, 
\end{align}

\begin{figure} 
\centering
\includegraphics[scale=0.55]{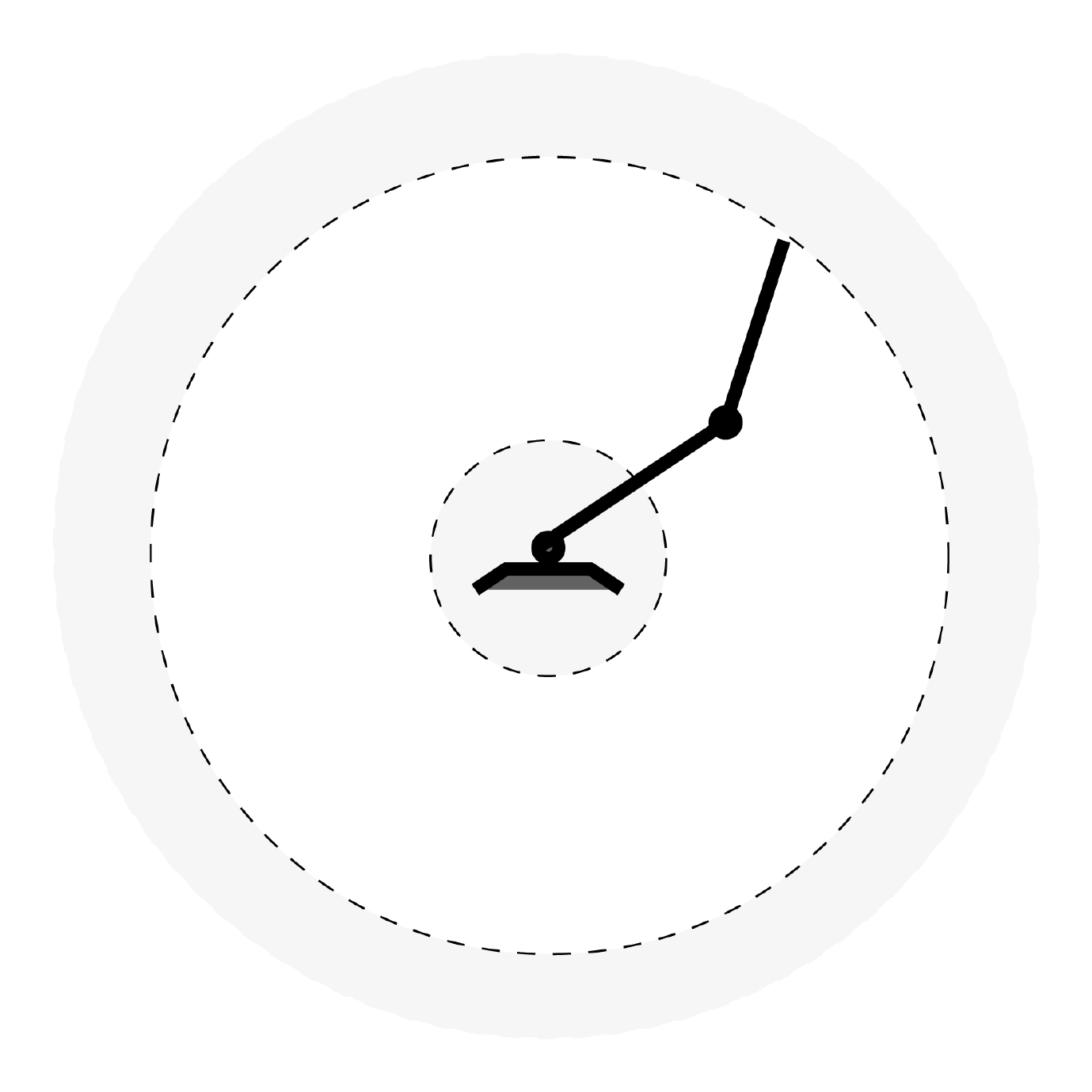}
\caption{Example of a problem which presents a singularity}
\label{example_2_arms}
\end{figure}

where  

\begin{align}
x &= l_1 {\rm cos}\theta_1 + l_2{\rm cos}(\theta_1 + \theta_2); \\
y &= l_1 {\rm sin}\theta_1 + l_2{\rm sin}(\theta_1 + \theta_2). 
\end{align}
\bigskip

Let us assume, for instance, to be in region $R_1$, thus we are integrating vector field $f_1$. 
We will consider the case in which $f_1$ cannot be evaluated outside $R_1 \cup \Sigma$. In this situation we will consider \emph{one-sided} Rosenbrock schemes that do not require the evaluation of the vector field $f_1$ outside $R_1 \cup \Sigma$. This approach has been proposed in \cite{Dieci.Lopez.2011} in the context of explicit Runge-Kutta methods. 
The idea is very simple: to give sufficient conditions for which the internal stages of a Rosenbrock method can be computed even in presence of a model singularity. For example, for the two-stages Rosenbrock method (\ref{eqn:Isabella3}), this means to give conditions for which $f_1$ can be evaluated in $x_0 + k_1$: this will ensure that $x_1$ can be computed. 
%

We will give sufficient conditions under which one-stage Rosenbrock methods approach the discontinuity from one side, and we will make use of the continuous extension (\ref{ISABELLA_cleavage}) in order to give one-sided conditions for the method (\ref{eqn:Isabella3}).

\emph{One-stage Rosenbrock: general case} \quad Assuming  $b_1 = 1$, one stage Rosenbrock method reads as:
\begin{equation}\label{Ros1}
\begin{array}{l}
x_1 = x_0 + k_1, \quad \textrm{where} \\
k_1 = \tau (I-\gamma \tau J)^{-1} f_1(x_0).
\end{array}
\end{equation}
As in \cite{Dieci.Lopez.2011}, we will assume that there is a positive constant $\delta$ such that 
\begin{equation}
\label{LanciaDelta}
h_{x}^{\top}(x)f_1(x)\geq \delta > 0, \quad \forall x \in R_1. 
\end{equation}
Now, if $x(t)$ is the solution of (\ref{system1}) in the region $R_1$, then \\
\[
\frac{d}{dt}h(x(t)) = h_{x}^{\top}(x(t))\frac{d}{dt}x(t) = h_{x}^{\top}(x(t))f_1(x(t)),
\]
and condition (\ref{LanciaDelta}) implies that the function $h$ monotonically increases along a solution trajectory in $R_1$ (close to $\Sigma$) until eventually the trajectory hits $\Sigma$ nontangentially. \\
We are going to give an analogous condition in the discrete environment of the numerical solution. 
We assume to be in the interval of the event, i.e. $x_0 \in R_1$ and $x_1 \in R_2$. In this interval, the continuous function $H(\sigma) := h(x_1(\sigma))$ changes its sign in $[0, \tau]$, then at least one $\eta \in [0, \tau]$ exists, such that $H(\eta) = 0$.\\
A sufficient condition for $\eta \in [0, \tau]$ to be the only root of the function $H$ is that the straight line segment $x_1(\sigma)$ intersects $\Sigma$ just once: a sufficient condition for this to be true is exactly the analogous of (\ref{LanciaDelta}), that is 
\begin{equation}
\label{derivata_num}
\frac{d}{d\sigma}h(x_1(\sigma))  = h_{x}^{\top}(x_1(\sigma))\frac{d}{d\sigma} x_1(\sigma) > 0, \quad \forall \sigma \in [0, \tau],
\end{equation}
where $x_1(\sigma)$ is defined from the (\ref{Ros1}) as 
\begin{equation}
x_1(\sigma) = x_0 + \sigma (I-\gamma \sigma J)^{-1}f_1(x_0) , \quad \forall \sigma \in [0, \tau], 
\end{equation}
and, obviously, 
\begin{equation}
\label{der}
\frac{d}{d\sigma}x_1(\sigma) = (I-\gamma \sigma J)^{-1}f_1(x_0) + \sigma \frac{d}{d\sigma} \big( (I-\gamma \sigma J)^{-1} \big)f_1(x_0), \quad \forall \sigma \in [0, \tau]. 
\end{equation}



Assuming again that $\rho(\gamma \sigma J) < 1$, for every $\sigma \in [0, \tau]$, we can write 

\begin{subequations}
\label{eqn:schema}
\begin{align}
(I - \gamma \sigma J)^{-1} &= \sum_{k = 0}^{\infty} (\gamma \sigma J)^{k}, \\
\frac{d}{d\sigma} (I - \gamma \sigma J)^{-1} &= \sum_{k = 1}^{\infty} k \sigma ^{k-1}(\gamma J)^{k}. \label{eqn:sub} 
\end{align}
\end{subequations}

Thus, collecting by the same power of $\sigma$, (\ref{der}) will become
\begin{equation}
\label{der3}
\frac{d}{d\sigma}x_1(\sigma) = f_1(x_0) + 2 \gamma \sigma J f_1(x_0) + 3 \gamma^2 \sigma^2 J^2 f_1(x_0) + ... ,
\end{equation}
while 
\begin{equation}
\label{expansion}
h_x^{\top}(x_1(\sigma)) = h_x^{\top}\big(x_0 + \sigma( I+ \gamma \sigma J + \gamma^2 \sigma^2 J^2+ .... ) f_1(x_0)\big); 
\end{equation}
truncating the power series up to the second power of $\sigma$, (\ref{expansion}) becomes 
\begin{equation}
\label{expansion2}
h_x^{\top}(x_1(\sigma)) \approx h_x^{\top}\big(x_0 + \sigma f_1(x_0) + \gamma \sigma^2 J f_1(x_0)),
\end{equation}
and then, just a Taylor expansion of (\ref{expansion2}) gives
\[
h_x^{\top}(x_1(\sigma)) \approx h_x^{\top} (x_0) + \big(\sigma f_1^{\top}(x_0) + \gamma \sigma^2 f_1^{\top}(x_0) J^{\top} \big) h_{xx}(x_0).
\]
With this approximation, and by virtue of (\ref{der3}), the scalar product in (\ref{derivata_num}) can be written as a power series expansion in $\sigma$: 
\begin{equation}
\label{power_series}
\begin{split}
h_{x}^{\top}(x_1(\sigma))\frac{d}{d\sigma} x_1(\sigma)  &= h_{x}^{\top}f_1 + \sigma \big[f_1^{\top} h_{xx}f_1 + 2\gamma h_x^{\top} Jf_1 \big] + \\
&\quad+ \sigma^2 \big[ 3 \gamma^2 h_x^{\top} J^2 f_1 + 2\gamma f_1^{\top} h_{xx}Jf_1 + \gamma f_1^{\top} J^{\top} h_{xx} f_1\big] \\
&\quad+ \sigma^3\big[ \cdots \big] + \ldots,
\end{split}
\end{equation}
where all the functions (both matrices and vectors) of the right-hand side are evaluated at $x_0$.\\ 
Let $a_k$ be the coefficient of $\sigma^k$ in the power series expansion (\ref{power_series}); hence inequality (\ref{derivata_num}) holds if and only if 
\[
\sum_{k = 0}^{\infty} a_k \sigma^k> 0, \quad \textrm{i.e.} \quad  \sum_{k = 0}^{2} a_k \sigma^k > - \sum_{k = 3}^{\infty} a_k \sigma^k. 
\]

In general this condition is in practice very difficult to verify. What we are going to do, in practice, is to truncate the power series up to the third term; thus, we will require just that the the sum of the first three terms of the power series in (\ref{power_series}) is greater than zero for $\sigma = \tau$. This condition can be summarized by the following proposition.
\begin{proposition}
Assuming that all the following function evaluations are performed at $x_0$, let there exist three constants $\rho_1$, $\rho_2$, and $\rho_3$, all greater than zero, and $\tau > 0$ and sufficiently small, such that
\begin{align} 
& i) \quad  h^{\top}_{x}f_1 > \rho_1, \\ 
& ii)  \quad f^{\top}_1h_{xx}f_1 + 2\gamma h^{\top}_x Jf_1 > - \rho_2, \\ 
& iii) \quad 3 \gamma^2 h^{\top}_x J^2 f_1 + 2\gamma f^{\top}_1 h_{xx}Jf_1 + \gamma f^{\top}_1 J^{\top} h_{xx} f_1 > - \rho_3, \\
& iv) \quad \rho_1 - \tau \rho_2 - \tau^2 \rho_3 > 0. 
\end{align}
Then the function $h(x_1(\sigma))$ is strictly increasing for every $\sigma \in [0, \tau]$. In particular, there exists a unique $\eta \in ( 0, \tau )$, such that $h(x_1(\eta)) = 0$.  
\end{proposition}


\emph{One-stage Rosenbrock: particular case} \quad Another approach is possible if we make the following assumption: \textbf{the matrix $(I-\gamma \sigma J)^{-1}$ is orthogonal}, i.e. $(I-\gamma \sigma J)^{-1} = (I-\gamma \sigma J)^{T}$. In this case, formula (\ref{der}) will become:
\begin{equation}
\label{deriv}
\frac{d}{d\sigma}x_1(\sigma) = f_1(x_0) -2\gamma \sigma J^{\top}f_1(x_0), \quad \forall \sigma \in [0, \tau]. 
\end{equation}
On the other hand,
\begin{equation}
\label{expansion4}
h_x^{\top}(x_1(\sigma)) \approx h_x^{\top}(x_0) + h_{xx}^{\top}(x_0) \big( \sigma f_1(x_0) - \gamma \sigma^2 J^{\top} f_1(x_0) \big), 
\end{equation}
hence, dropping the argument of the functions whenever it is $x_0$, the scalar product (\ref{derivata_num}) becomes 
\begin{equation}
\begin{split}
\label{quasifinito}
h_{x}^{\top}(x_1(\sigma))\frac{d}{d\sigma} x_1(\sigma) =& \quad h_{x}^{\top}f_1 + 
\sigma f_1^{\top} h_{xx}^{\top}f_1 - \\
& - \gamma \sigma^2 f_1^{\top} h_{xx}^{\top}J^{\top}f_1 - 2\gamma \sigma h_{x}^{\top}J^{\top}f_1 -\\
& - 2\gamma \sigma^2 f_1^{\top} h_{xx}J^{\top}f_1 + 2\gamma^2 \sigma^3f_1^{\top}Jh_{xx}J^{\top}f_1. 
\end{split}
\end{equation}
Truncating this product to the second power of $\sigma$, we get
\begin{equation}
\begin{split}
h_{x}^{\top}(x_1(\sigma))\frac{d}{d\sigma} x_1(\sigma) =& h_{x}^{\top}f_1 + 
 \sigma \big( f_1^{\top} h_{xx}^{\top}f_1 - 2\gamma h_{x}^{\top}J^{\top}f_1 \big) - \\
& - \sigma^2 \big( 2\gamma f_1^{\top} h_{xx}J^{\top}f_1 + \gamma f_1^{\top} h_{xx}^{\top}J^{\top}f_1 \big).
\end{split}
\end{equation}

Finally, sufficient conditions for (\ref{derivata_num}) to be satisfied are 
\begin{align*} 
& i) \quad  h_{x}^{\top}(x_0)f_1(x_0) > \delta_1, \\ 
& ii)  \quad f_1^{\top}(x_0) h_{xx}^{\top}(x_0)f_1(x_0) - 2\gamma h_{x}^{\top}(x_0)J^{\top}f_1(x_0) > -\rho_1, \\ 
& iii) \quad 2\gamma f_1^{\top}(x_0) h_{xx}(x_0)J^{\top}f_1(x_0) + \gamma f_1^{\top}(x_0) h_{xx}^{\top}(x_0)J^{\top}f_1(x_0) > -\rho_2;\\
& iv) \quad \delta_1 - \tau \rho_1 - \tau^2 \rho_2 > 0. 
\end{align*}
\smallskip
\emph{Two-stages Rosenbrock: the continuous extension approach} \quad In the previous section we have presented a two-stages Rosenbrock method (\ref{eqn:Isabella3}).
For this method, there are two possibilities: 

\begin{description} 
\item[ 1.a) ]  $h(x_0 + k_1) \leq 0$, \quad and \quad $h(x_1) \geq 0$,  \label{stage_below}
\item[ 1.b) ]  $h(x_0 + k_1) > 0$. 
\end{description}
\begin{figure} 
\center
\includegraphics[scale=0.50]{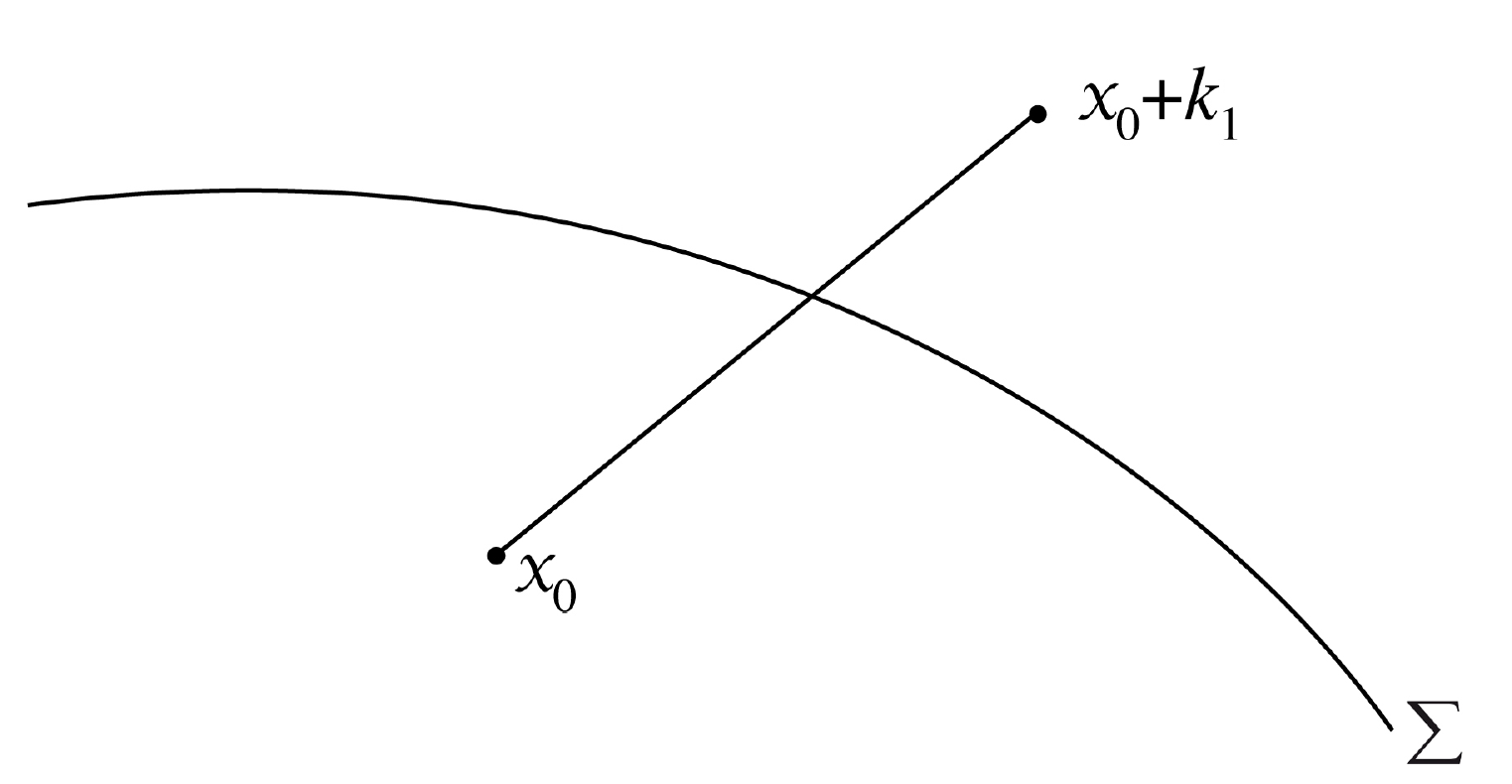}
\caption{Internal stage $(x_0 + k_1) \in R_2$}
\label{Ros_2stages_bSimon}
\end{figure}
In case {\bf 1.b)} -see Figure \ref{Ros_2stages_bSimon}-, since $h(x_0 + k_1)>0$, we cannot properly compute $x_1$; in this situation, a step reduction is needed, in such a way to find the value $\overline{\sigma}$ such that $h(x_0 + k_1(\overline{\sigma})) = 0$. Then, if $x_1(\overline{\sigma})$ is above $\Sigma$, we are back to case 1.a), with step size $\overline{\sigma}$, otherwise we continue integrating. \\
\begin{figure} 
\center
\includegraphics[scale=0.50]{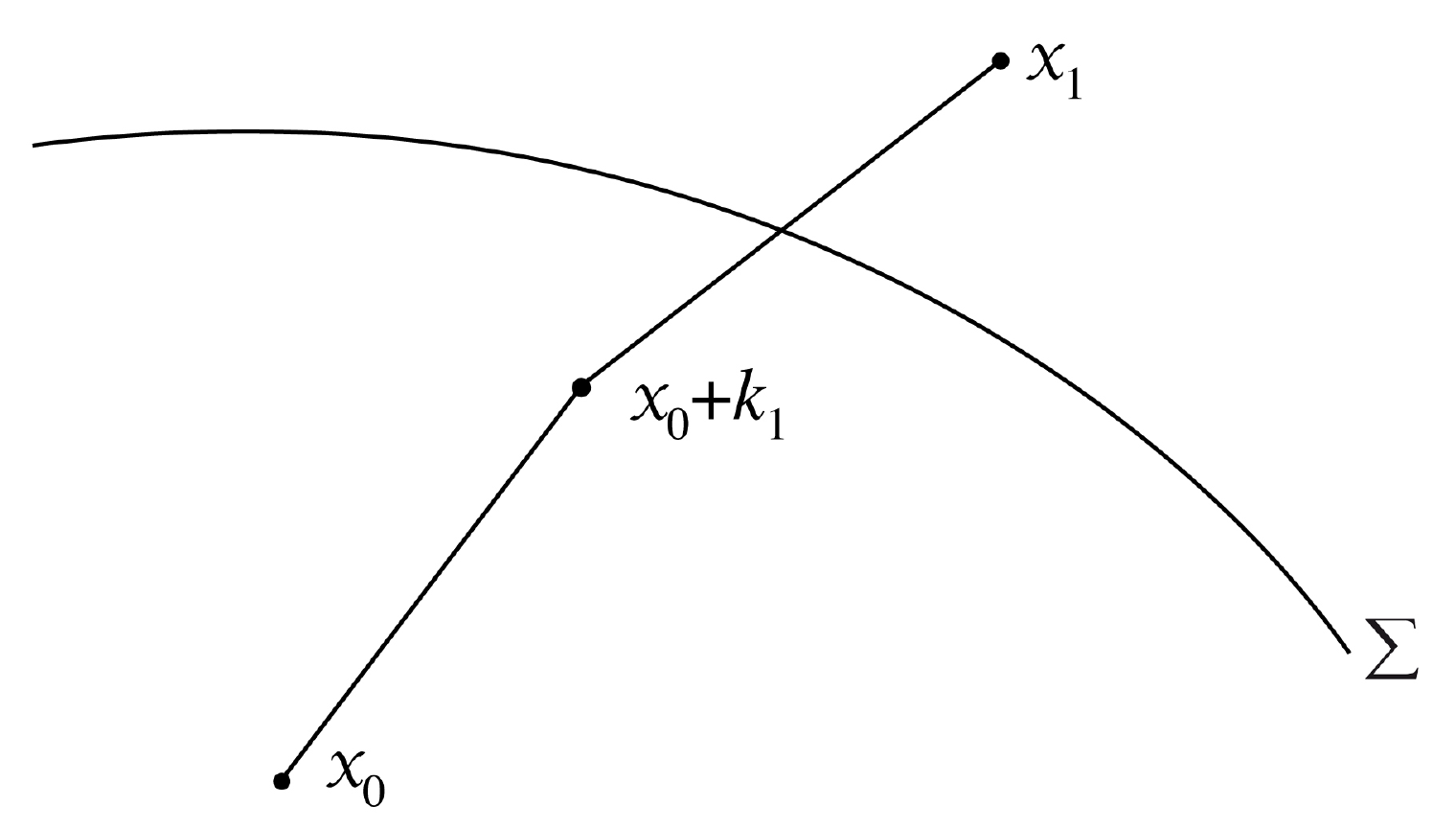}
\caption{Internal stage $(x_0 + k_1) \in R_1$}
\label{Ros_2stages_aSimon}
\end{figure}
In case {\bf 1.a)} -see Figure \ref{Ros_2stages_aSimon}- the sufficient condition which guarantees the uniqueness of $\overline{\sigma}$ such that $h(x_1(\overline{\sigma})) = 0$ is (\ref{derivata_num}). 
Now, the problem is computing $\frac{d}{d\sigma}x_1(\sigma)$. \\
\smallskip

In (\ref{ISABELLA_cleavage}) we have presented a continuous extension of method (\ref{eqn:Isabella3}). Since this continuous extension has the same order of the method, in the discontinuity interval we could confuse the method itself with its continuous extension. Thus, we could replace condition (\ref{derivata_num}) with its analogous
\begin{equation}
\label{derivata_num2}
\frac{d}{d\theta}h(X_1(\theta))  = h_{x}^{\top}(X_1(\theta))\frac{d}{d\theta} X_1(\theta) > 0, \quad \forall \theta \in [0, 1].
\end{equation}
The great advantage with respect to (\ref{derivata_num}) is in the computation of $\frac{d}{d\theta} X_1(\theta)$, since the internal stages $k_1$ and $k_2$ do not depend neither on $\sigma$ nor on $\theta$. Thus the derivative of $X_1(\theta)$ with respect to $\theta$ is computed just by deriving a couple of second order polynomials! \\
%
But we have just argued that condition (\ref{derivata_num}) is essentially equivalent to (\ref{derivata_num2}). Thus we are able to state the following theorem.

\begin{theorem}
\label{theo_2stages_Rosenbrock}
Consider the case {\bf 1.a)}, and assume that $h(x_0 + k_1) \leq 0$ (we know that, with this approach, $k_1$ does not depend on $\theta$). Denote by $c$ the constant $\frac{1}{2(1-2\gamma)}$ in the definition of continuous extension (\ref{ISABELLA_cleavage}). Moreover, assume that there exist constants $\delta_1 > 0$ and $\rho_1 >0$, and let $\tau >0$, small enough so that, for the continuous extension (\ref{ISABELLA_cleavage}), the following conditions hold:
\begin{enumerate}
                 \item $h_x^{\top}(X_1(\theta)) c [(2-6\gamma)k_1 - 2\gamma k_2] \geq \delta_1$, \quad for all \quad $\theta \in [0, 1]$;
                 \item $h_x^{\top}(X_1(\theta)) 2c (k_1 + k_2) \geq -\rho_1$, \quad for all \quad $\theta \in [0, 1]$;
                 \item $\delta_2 - \tau \rho_2 \geq 0$. 
                 \end{enumerate}
Then, the function $h(X_1(\theta))$ is strictly increasing for $\theta \in [0, 1]$. In particular, there exists a unique $\eta$ such that $h(X_1(\eta)) = 0$. 
\end{theorem}
\begin{proof}
For every $\theta \in [0, 1]$,
\[
\frac{d}{d\theta}h(X_1(\theta))  = h_{x}^{\top}(X_1(\theta))\frac{d}{d\theta} X_1(\theta) > 0.
\]
Now, given the continuous extension (\ref{ISABELLA_cleavage}), 
\begin{equation*} 
\begin{split}
\frac{d}{d\theta} X_1(\theta) &=  c \left( \frac{d}{d\theta} b_1(\theta) k_1 + \frac{d}{d\theta} b_2(\theta)k_2 \right) =  \\
 &= c\left( 2\theta + 2 - 6\gamma \right) k_1 + c \left( 2\theta -2\gamma \right) k_2,  
\end{split}
\end{equation*}
it results that 
\begin{equation*} 
\begin{split}
h_{x}^{\top}(X_1(\theta))\frac{d}{d\theta} X_1(\theta) &= h_{x}^{\top}(X_1(\theta)) c \big( \left( 2-6\gamma \right)k_1 -2\gamma k_2 \big)  \\
 &+ h_{x}^{\top}(X_1(\theta)) 2c\gamma \left(k_1 + k_2 \right).
\end{split}
\end{equation*}
Finally, using hypotheses 1., 2. and 3., we get that 
\[
\frac{d}{d\theta} X_1(\theta) \geq \delta_1 - \tau \rho_1 >0.
\]
\end{proof}

%
\section{Discontinuous singularly perturbed systems}
Discontinuous singularly perturbed systems are very interesting since they combine both the features of singularly perturbed differential systems, and the ones of the differential equations with discontinuous right-hand side. \\
We are going to study a singularly perturbed system where the discontinuity involves just the derivative of the 'slow' component, i.e.:
 
\begin{equation}
\label{sing_pert_sys}
\left\{
\begin{array}{rl} 
\dot{y} &=f(y,z); \\
\epsilon\dot{z} &=g(y,z). 
\end{array}
\right.
\end{equation}
Here $\epsilon$ is a positive small parameter, and the right-hand side of the first equation is defined in the following way: 
\begin{equation}\label{system4}
f(y,z)= \left\{
\begin{array}{rl}
 f_1(y,z)& \mbox{when } h(y,z) < 0 \\
 f_2(y,z)& \mbox{when } h(y,z) > 0 \\
\end{array},
\right.
\end{equation}
where $y = y(t) \in \R^s$ and $z = z(t) \in \R^m$, and $h:\R^s \times \R^m \rightarrow \R$ is the event function presented in the beginning. \\
Here, we remind some basic concepts of singularly perturbed systems,
\begin{em}
in the smooth case. 
\end{em}
 For a complete coverage, see \cite{OM}, or \cite{Smith}, or also \cite{Hairer.Wanner.2008}.
%
 We know that the solution of this singularly perturbed system can be written as the superposition of the \emph{outer solution}, a smooth function of the independent variable $t$, that approximates the exact solution for values of $t$ well away from the initial instant $t = 0$, and the \emph{initial layer correction}, a rapidly decaying function of the stretched time $\frac{t}{\epsilon}$, which plays an important role just in the initial $\epsilon$-thick boundary layer: outside of it is negligible. In symbols:  
 \begin{equation}
\label{sing_pert_solution}
\left\{
\begin{array}{rl} 
y(t,\epsilon) &=Y(t,\epsilon) + \epsilon \eta\left(\frac{t}{\epsilon}, \epsilon \right) \\
z(t,\epsilon) &=Z(t,\epsilon) + \zeta \left(\frac{t}{\epsilon}, \epsilon \right) 
\end{array}
\right.
\end{equation} 
where $Y(t,\epsilon)$ and $Z(t,\epsilon)$ are the outer solutions of slow and fast variable, respectively, while  $\eta\left(\frac{t}{\epsilon}, \epsilon \right)$ and $\zeta \left(\frac{t}{\epsilon}, \epsilon \right)$ are the initial layer corrections of the slow and fast variable, respectively.\\
\subsection{Singular perturbation {\rm vs} DAE} \label{DAE_Approximation}
When dealing with discontinuous singularly perturbed systems, the choice of many authors -mostly when the system is linear in both or either one of the variables- is to consider the \emph{reduced system}, i.e. the system 
\begin{subequations}
\label{eqn:DAE}
\begin{align}
\dot{y} &= f(y,z), \\
0 &= g(y,z). \label{eqn:DAE2} 
\end{align}
\end{subequations}
obtained by (\ref{sing_pert_sys}) by setting $\epsilon = 0$ (see, for example, \cite{AlvarezGallego.SilvaNavarro, Heck, Su.WuChung}). \\
As a matter of fact, this choice makes sense in particular when variable $z$ can be globally expressed as a function of $y$, in equation (\ref{eqn:DAE2}). Indeed, the fact that equation $0 = g(y,z)$ admits a global isolated solution with respect to $z$ of the form $z = g_0(y)$, is a pretty strong assumption, but it considerably reduces both the dimension and the stiffness of the problem, since by this hypothesis, the DAE (\ref{eqn:DAE}) 
corresponding to the system (\ref{sing_pert_sys}) is equivalent to the so-called \emph{reduced-order model}:
 
\begin{equation} \label{Reduced_Order_Model}
y' = f(y, g_0(y)).
\end{equation}

Under the further assumption that $\textrm{Re} {\rm Spec} \frac{\partial g}{\partial z}$, the fundamental problems are: 
\begin{enumerate}
\item is this a good approximation?
\item does the qualitative behaviour of the starting system (\ref{sing_pert_sys}) hold, when considering the reduced system (\ref{Reduced_Order_Model})?
\end{enumerate}
\emph{Answers}
\begin{enumerate}
\item First of all we assume that the event is ``far'' enough from the initial boundary layer, i.e. that the event occurs when the transient phase is over. In this way, we can reasonably neglect the fast decaying term both in the slow and in the fast variable, namely, in the notation of (\ref{sing_pert_solution}), $\eta\left(\frac{t}{\epsilon}, \epsilon \right)$ and $\zeta \left(\frac{t}{\epsilon}, \epsilon \right)$, respectively. Moreover, we assume that the events are far enough from each other.\\
These assumptions guarantee that the \emph{fast} system is smooth, i.e. the discontinuity surface is not intersected by the solution in the initial layer. 
Of course, by smooth singular perturbation theory, we know that the equation (\ref{Reduced_Order_Model}) is just an O($\epsilon$) approximation of (\ref{sing_pert_solution}). However, we have to notice that, since $\epsilon$ is a small parameter (for instance, in some real models it is $\epsilon \approx 10^{-4}$), very often an O($\epsilon$) approximation is sufficiently accurate, in applications. \\
For instance, the example proposed in \cite{Heck} decouples the system in a slow and a fast subsystem, solves them separately, and matches the solution in the border of $\epsilon$-thick boundary layer. 
\medskip

\item The answer, in general, is negative. A nice example can be found in \cite{Cardin.DaSilva.Teixeira}. Considering the system 
\begin{equation}
\label{ex_Teixeira}
\begin{array}{l} 
\left [ \begin{array}{c} y'_1   \\
  y'_2  \\
   \epsilon z' \\
\end{array}\right ]

 = \left [ \begin{array}{c} - {\rm sign}(2z - y_1)   \\
  - y_1 - y_2  \\
    y_1 - z \\
\end{array}\right ],
\end{array}
\end{equation}
simple computations show that sliding cannot occur in the system (\ref{ex_Teixeira}), because conditions for attractive sliding would be $y > \frac{3}{2} \epsilon$ and $y < \frac{1}{2} \epsilon$, which is impossible, since $\epsilon > 0$. \\
On the other hand, plugging $\epsilon = 0$ in system (\ref{ex_Teixeira}), we observe that the  reduced system exhibits attractive sliding in all points of the switching manifold.  \\
Another interesting example is presented in \cite{Sieber.Kowalczyk}. Here the discontinuous perturbed system 
\begin{equation}\label{esempio_Kowalczyk}
\left\{ \begin{array}{ccc}
x'= & - {\rm sign} [\theta x+(1-\theta)y] \\ 
\epsilon y'\hfill &=x-y \hfill 
\end{array}, \right. 
\end{equation}
depends on a parametere $\theta$ that can be greater or less than zero. It can be shown that the reduced order model has a stable equilibrium point in the origin $(0,0)$, whereas the perturbed system presents, when $\theta < 0$, an exponentially stable periodic orbit around the origin, switching between the two different vector fields $F_1$ and $F_2$. This system is the object of the numerical tests presented in the last section. 
\end{enumerate}

\subsection{Sliding or crossing}
An interesting feature in the treatment of discontinuous singularly perturbed system is the study of conditions for sliding or crossing. Assume we are on the switching manifold, i.e. that $x \in \Sigma$. From now on, each function evaluation will be accomplished in $x \in \Sigma$, so we will drop the argument of each function.\\
From Fillipov's theory, we know that the only two situations that guarantee the uniqueness of a solution when approaching the discontinuity surface are the following ones: crossing and attractive sliding. Crossing simply means that the state vector, coming from one of the vector fields (for instance, $f_1$) ``hits'' the surface $\Sigma$ and crosses it instantly. Attractive sliding means that the state vector is forced to move along $\Sigma$ with a yet to be defined vector field. For a complete coverage, see \cite{Acary.Brogliato, diBernardo-book, Dieci.Lopez.survey, Filippov, Utkin}. Instead, for the definition of the sliding vector field on the intersection of surfaces, see \cite{Dieci.Lopez.2008/5, Dieci.Lopez.2009, Dieci.Lopez.Elia, Piiroinen.Kuznetsov}. \\
 We know that, if $n = n(x) = \nabla h (x)$, in the notation of system (\ref{system1}),  
\begin{itemize}
\item
\emph{crossing} occurs if $(n^T f_1) (n^T f_2) > 0$,
\item \emph{sliding} (both attractive and repulsive) 
occurs if $(n^T f_1) (n^T f_2) < 0$.
\end{itemize}
Now, let us rewrite these conditions in the singularly perturbed case, i.e. in case of system (\ref{sing_pert_sys})-(\ref{system4}). Define 
\[
F_1=
\begin{bmatrix}
f_1   \\
\frac{g}{\epsilon}  
\end{bmatrix},
\qquad  F_2=
\begin{bmatrix}
f_2   \\
\frac{g}{\epsilon} 
\end{bmatrix}
\]
Thus, condition for sliding, (whether it is attractive or repulsive) is, naturally, 
\[
(n^T F_1) (n^T F_2) < 0,
\]
which means, 
\[
\left( h_y f_1 + \frac{1}{\epsilon} h_y g \right) \left( h_y f_2 + \frac{1}{\epsilon} h_y g \right) < 0,
\]
and reordering with respect to the power of $\epsilon ^ {-1}$,
\[
\left( h_y f_1 \right) \left( h_y f_2 \right) + \frac{1}{\epsilon} \Big( \left( h_y f_1 \right)  \left( h_z g \right) + \left( h_y f_2 \right) \left( h_z g \right) \Big) + \frac{1}{\epsilon ^2}  \left( h_z g \right)^2   < 0,
\]
from which we get, just multiplying for $\epsilon ^ 2$, the following inequality: 
\begin{equation} \label{disequazione}
\left( h_y f_1 \right) \left( h_y f_2 \right) \epsilon ^2 + \Big( \left( h_y f_1 \right)  \left( h_z g \right) + \left( h_y f_2 \right) \left( h_z g \right) \Big) \epsilon + \left( h_z g \right)^2   < 0.
\end{equation}
In this way we have just to examine an algebraic inequality in $\epsilon$. But we know in advance that $0 < \epsilon <<1$. We denote by $A$ the coeffient of $\epsilon ^2$, by $B$ the coefficient of $\epsilon$, and $C^2 = \left( h_z g \right)^2$. Naturally, $A$, $B$ and $C$ are real numbers, since they are just sums and products of scalar products. Thus, (\ref{disequazione}) can be rewritten in the following form 
\begin{equation} \label{per_fare_prima}
A \epsilon ^2 + B \epsilon + C^2 < 0.
\end{equation}
Now, it becomes clear that, in the limit for $\epsilon \rightarrow 0$, the latter inequality is not satisfied. This means, roughly speaking, that sliding is less ''likely`` than crossing, for the singularly perturbed system (\ref{sing_pert_sys}). In the following, sufficient conditions are given for which sliding and crossing occur.
\begin{proposition}
Sliding occurs if the following conditions are fullfilled:\\
\[       
A < 0, \quad B^2 - 4 AC^2 \neq 0, \quad \textrm{and} \quad \epsilon > \frac{-B - \sqrt{B^2 + 4 |A| C^2} }{2A} 
\]                 
                       
\end{proposition}
\begin{proof}
If $A < 0 $, then $B^2 - 4 AC^2 > 0$; by Descartes' rule of signs, there will be necessarily a positive root and a negative one. Thus, inequality (\ref{per_fare_prima}) holds for $\epsilon < \epsilon_1 :=  \frac{-B + \sqrt{B^2 + 4 |A| C^2} }{2A}$, and $\epsilon > \epsilon_2 :=  \frac{-B - \sqrt{B^2 + 4 |A| C^2} }{2A}$. But $\epsilon$ has to be greater than $0$, so the only choice is $\epsilon > \epsilon_2$ (since $\epsilon_1 < 0$). 
\end{proof}
Crossing occurs, from Filippov theory, if an analogous of (\ref{disequazione}) holds, with the opposite sign, i.e.
\begin{equation} \label{disequazione2}
\left( h_y f_1 \right) \left( h_y f_2 \right) \epsilon ^2 + \Big( \left( h_y f_1 \right)  \left( h_z g \right) + \left( h_y f_2 \right) \left( h_z g \right) \Big) \epsilon + \left( h_z g \right)^2 > 0,
\end{equation}
which, in the notation of (\ref{per_fare_prima}), becomes 
\begin{equation} \label{per_fare_prima_2}
A \epsilon ^2 + B \epsilon + C^2 > 0.
\end{equation}
\begin{proposition}
Crossing occurs, in system (\ref{sing_pert_sys}), if one of the following conditions is satisfied: 
\begin{enumerate}
                 \item $A > 0$ and $B^2 - 4 AC^2 < 0$; 
                 \item $A > 0$ and $B > 0$, assuming also $B^2 - 4 AC^2 \geq 0$. 
       \end{enumerate}
\end{proposition}        
\begin{proof} \quad
\begin{enumerate}
\item Under conditions in 
1, polynomial in (\ref{per_fare_prima_2}) has no real roots, and assumes values grater than zero for every $\epsilon$ in $\R$.\\
\item Under conditions in 
2, from the Descartes' rule of signs, we know that, in the polynomial of inequality (\ref{per_fare_prima_2}), if the signs of coefficients do not change, and if $B^2 - 4 A C^2 \neq 0$, then the roots of the polynomial, denoted by $\epsilon_1$ and $\epsilon_2$, are both negative. Naturally, inequality (\ref{per_fare_prima_2}) holds in the intervals $(- \infty, \epsilon_1)$, and $(\epsilon_2, \infty)$. But we know in advance that $\epsilon$ has to be such that $0 < \epsilon << 1$, thus inequality (\ref{per_fare_prima_2}) is satisfied. 
\end{enumerate}
\end{proof}

\begin{remark}
Looking at the left hand-side of inequality (\ref{disequazione2}), we observe that, if the event function $h$ does not depend on $y$ (i.e. if $h$ is such that $h = h(z)$), then the only allowed behaviour of system (\ref{sing_pert_sys}) is crossing. \\
Conversely, if the switching function $h$ depends just on $y$, i.e. $h = h(y)$, then necessarily the terms $B$ and $C^2$ are zero. In this case condition for sliding or crossing is driven uniquely by the sign of $A$.  
\end{remark}

\begin{example}
As an example of the latter remark, let us consider the following system, which is a modified, discontinuous version of an example given in \cite{Ostermann}. 
\begin{equation} \label{Ost_modified}
\left\{
\begin{array}{rl}
 \dot{y_1} =& z \\ 
 \dot{y_2} =& -{\rm sign}(y_1) y_1 \\
 \epsilon \dot{z} =& y_2 - z - \epsilon y_1 
\end{array}.
\right.
\end{equation}
Obviously, switching function $h$ does not depend on $z$, since $h = h(y) = y_1$. Thus, both $B$ and $C$ are zero, and, being $A = z^2 > 0$, then we can state that system (\ref{Ost_modified}) crosses the switching manifold at the event point. 
\end{example}

\subsection{Numerical issues}
In subsection \ref{DAE_Approximation} we have discussed if it is convenient or not to approximate the system (\ref{sing_pert_sys}) with its reduced order model (\ref{eqn:DAE}). From a strictly numerical point of view, it is well known that the great advantage of Rosenbrock methods with respect to implicit Runge-Kutta methods is the \emph{linearity}. Now, if we integrated system (\ref{Reduced_Order_Model}) by a Rosenbrock method, we would lose this advantage (see \cite{Hairer.Wanner.2008}). For this reason, we have chosen to integrate the original system (\ref{sing_pert_sys}), and not the corresponding differential algebraic system. \\
\medskip
\scriptsize
 \begin{table}[htbp]
\label{table1}

\vspace{3mm}
    \centering

    {
        \begin{tabular}{cccc}
         \hline
         \hline

  stepsize $\tau$              &  $\epsilon$   &  Global Error & Reduction Factor   \\
    \hline
    \hline &&&\\
    1E-3      & 1E-2 &    2.180627E-4  & \\                            
    0.5E-3    & 1E-2 & 1.084872E-4      & $ 2.0100 $ \\  
    0.25E-3   & 1E-2 &  5.426686E-5   &   $1.9991 $  \\  
    0.125E-3  & 1E-2 &  2.713762E-5   &  $ 1.996$  \\  
    0.0625E-3 & 1E-2 &  1.356825E-5   &  $ 1.995$  \\  
    
\hline
\hline
\\[0.2 cm]
    1E-5      & 1E-3 &    2.202832E-4  & \\                            
    0.5E-5    & 1E-3 & 1.102479E-4   & $ 1.9980 $ \\  
    0.25E-5   & 1E-3 &  5.526992E-5  &   $1.9947 $  \\  
    0.125E-5  & 1E-3 &  2.765977E-5 &  $ 1.9982$  \\  
    0.0625E-5 & 1E-3 &  1.382432E-5 &  $ 2.0008$  \\  
\hline
\hline
\\[0.2 cm]
     1E-6          & 1E-4 &    2.202745E-4  & \\                            
    0.51E-6    & 1E-4 & 1.107174E-4  & $ 1.9895 $ \\  
    0.25E-6    & 1E-4 & 5.545491E-5  &   $1.9965 $  \\  
    0.1251E-6  & 1E-4 &  2.770196E-5 &  $ 2.0018$  \\  
    0.0625E-6  & 1E-4 &  1.385162E-5 &  $ 1.9999$  \\  
    \hline
    \hline
        \end{tabular}}
        \caption{Global Error and Reduction Factor for first order Rosenbrock method (\ref{eqn:ros1}), for different values of parameter $\epsilon$}
        \label{ContExtRos1}
\end{table}

\normalsize
We have considered the system (\ref{esempio_Kowalczyk}) with $\theta = -0.9$.
First of all, we have considered one-stage Rosenbrock method
\begin{subequations}
\label{eqn:ros1}
\begin{align}
x_{1}  &=\ x_{0}+ k_1  \\
(I-\tau J)\ k_1 \ &=\ f_1(x_0). \label{eqn:rosk} 
\end{align}
\end{subequations}

According to the numerical experiments, we observe that the latter method, applied to system (\ref{esempio_Kowalczyk}), does not lose its order, even reducing the value of the parameter $\epsilon$. Now, the event is found by the bisection technique applied to the continuous extension $X_1(\theta) = x_0 + \theta k_1$, where $k_1$ is the same as in (\ref{eqn:rosk}) and $0 \leq \theta \leq 1$. Once the first event is localized with the requested tolerance (i.e the state vector is on the sliding surface, with a good approximation), we have computed the global error in that point, for different values of the step size. Halving the step size, we have also provided the reduction factor, i.e. the ratio between the global error obtained with step size $\tau$ and the one we got using stepsize $\frac{\tau}{2}$. This confirms that scheme (\ref{eqn:ros1}) behaves like a first order method. 
\smallskip

Similar results are provided for the second order Rosenbrock method (\ref{eqn:Isabella3}), together with its continuous extension (\ref{ISABELLA_cleavage}), used in the context of event location. Table \ref{MariaMazza} confirms that the continuous extension (\ref{eqn:Isabella3}) is a second order interpolant for the corresponding method. 
The computational saving of using the continuous extension is definetely more evident for the second order method with respect to first order one. \\

 \scriptsize
 \begin{table}[htbp]

\vspace{3mm}
    \centering

    {
        \begin{tabular}{cccc}
         \hline
         \hline

  stepsize $\tau$              &  $\epsilon$   &  Global Error & Reduction Factor   \\
    \hline
    \hline &&&\\
     1E-3                         & 1E-2 &    7.880118E-5  & \\                            
    0.51E-3   & 1E-2  &    2.079523E-5 & 3.7893    \\  
    0.25E-3   & 1E-2  &    5.760348E-6 & 3.6100    \\  
    0.125E-3  & 1E-2  &    1.434942E-6 & 4.0143    \\  
    0.0625E-3 & 1E-2  &    3.581187E-7 & 4.0068    \\
\hline
\hline
\\[0.2 cm]
      1E-5       & 1E-3  &    9.336405E-7  & \\                 
      0.5E-5     & 1E-3  &    2.331221E-7  & 4.0049    \\  
      0.25E-5    & 1E-3  &    5.825643E-8  & 4.0016    \\  
      0.125E-5   & 1E-3  &    1.456912E-8  & 3.9986    \\  
      0.0625E-5  & 1E-3  &    3.648825E-9  & 3.9928    \\
\hline
\hline
\\[0.2 cm]
      1E-5         & 1E-4 &    8.310706E-5  & \\       
    0.5E-5         & 1E-4 &    2.125030E-5  & 3.9108    \\  
    0.25E-5        & 1E-4 &    5.678856E-6  & 3.7420    \\  
    0.125E-5       & 1E-4 &    1.427827E-6  & 3.9772    \\  
    0.0625E-5      & 1E-4 &    3.658980E-7  & 3.9022 \\
    \hline
    \hline
          
        \end{tabular}}
        \caption{Global Error and Reduction Factor for second order Rosenbrock (\ref{eqn:Isabella3}) for different values of $\epsilon$}
        \label{MariaMazza}
\end{table}

\normalsize

\section{Conclusion and future work}
In this paper we have studied some issues about the applications of Rosenbrock methods in the context of discontinuous differential systems. We focused on conditions for \emph{one-sided} Rosenbrock methods, and we showed the convenience of using the continuous extension in the context of event location. \\
We could carry on our research by studying the integration of sliding vector field by means of any implicit and semi-implicit Runge-Kutta scheme. \\
Of particular interest cuold be also the study of second order differential equations involving discontinuity just in the second derivative. These problems arise frequently in impact mechanics. On the other hand, there is a huge literature on \emph{smooth} second-order differential equations, which could be precious in this context (see, for example, \cite{Dambrosio.Esposito.Paternoster.2011, Dambrosio.Esposito.Paternoster.2012}).



\end{document}